\documentclass[11pt]{amsart}

\usepackage{amsmath,amssymb,mathtools}
\usepackage{booktabs}
\usepackage{array}
\usepackage{microtype}
\usepackage[hidelinks]{hyperref}
\hypersetup{
  pdftitle={Prime multipliers of order ten: norm spectra, local charts, and selective lifting},
  pdfauthor={Yuhu Wang and Xiao Zhang},
  pdfkeywords={Fourier matrix, principal minor, square-free order, cyclotomic norm, local prime ideals, selective lifting}
}

\newtheorem{theorem}{Theorem}[section]
\newtheorem{proposition}[theorem]{Proposition}
\newtheorem{lemma}[theorem]{Lemma}
\newtheorem{corollary}[theorem]{Corollary}
\theoremstyle{remark}
\newtheorem{remark}[theorem]{Remark}
\newtheorem{question}[theorem]{Question}

\newcommand{\F}{\mathcal{F}}
\newcommand{\Z}{\mathbb{Z}}
\newcommand{\Q}{\mathbb{Q}}
\newcommand{\Pcal}{\mathsf{P}}
\newcommand{\ord}{\operatorname{ord}}

\newcommand{\Nm}{\operatorname{N}}
\newcolumntype{L}[1]{>{\raggedright\arraybackslash}p{#1}}

\title[Prime multipliers of order ten]
{Prime multipliers of order ten:\\
norm spectra, local charts, and selective lifting}

\author{Yuhu Wang}
\address{Yuhu Wang, Department of Mathematics, Fudan University,
Shanghai 200433, China}
\author{Xiao Zhang}
\address{Xiao Zhang, Department of Mathematics, Fudan University,
Shanghai 200433, China}

\date{August 2026}


\begin{document}

\begin{abstract}
Write \(\Pcal(N)\) for the assertion that every principal minor of the
Fourier matrix of order \(N\) is nonzero.  The square-free principal-minor
conjecture was previously known for a few uniform small-multiplier families,
including several families with two prime factors; broader higher-factor
results were non-uniform, apart from isolated exact verifications.  We prove
a near-complete prime-multiplier theorem for the composite base \(10\):
\(\Pcal(10p)\) holds for every prime \(p\notin\{2,5,11\}\).

The proof begins with the complete cyclotomic norm spectrum of the principal
minors of the order-ten Fourier matrix.  Its rational-prime support is
\(\{2,3,5,11,31\}\).  Ordinary finite-characteristic lifting settles the
norm-safe multipliers.  A known small-multiplier theorem handles one
norm-exceptional case; another is settled by retaining the prime-ideal chart
in which each carrier degenerates.  This yields a flag-selective lifting
lemma and an active-rank norm budget.

The chart analysis also isolates the limitation at the remaining square-free
exception: active carriers can cover every local chart, so the first-order
argument stops.  We discuss this obstruction, the non-uniformity of fixed-base
  lifting, and the difficulties in passing to general square-free orders.  All
finite calculations are exact and have been independently confirmed.
\end{abstract}

\maketitle

\section{Overview and main results}\label{sec:overview}

For \(N\geq2\), put
\[
  \F_N=(\omega_N^{ij})_{i,j\in\Z_N},
  \qquad \omega_N=e^{-2\pi i/N},
\]
and adopt the convention
\(\det\F_N[\varnothing,\varnothing]=1\).  Write \(\Pcal(N)\) for the
assertion that every principal minor of \(\F_N\) is nonzero.

The prime-order all-minors theorem and the composite-order principal-minor
problem have different arithmetic boundaries.  The conjectural boundary for
the latter is
\[
  \Pcal(N) \quad\Longleftrightarrow\quad N\text{ is square-free};
  \tag{1.1}\label{eq:squarefree-conjecture}
\]
it was formulated from two directions involving principal Fourier minors and
woven or hierarchical exponential Riesz bases
\cite{CabrelliMolterNegreira2025,CarageaLee2024,
CarageaLeeMalikiosisPfander2025}.  The present paper studies the conjecture
through the complete prime-multiplier spectrum of one composite base rather
than through a single target order.  It is intended as a sequel to
Gu--Zhou--Wang \cite{GuZhouWang2026}: that article treated two fixed orders
through direct finite-field certificates, whereas here we classify prime
multipliers of a fixed composite base and isolate the local obstruction left
by ordinary lifting.

\begin{theorem}[Prime-multiplier near-classification]\label{thm:main}
Let \(p\) be prime.  If
\[
  p\notin\{2,5,11\},
\]
then every principal minor of \(\F_{10p}\) is nonzero.  If \(p\in\{2,5\}\),
then \(10p\) is not square-free and \(\F_{10p}\) has a zero principal minor.
The square-free case \(p=11\), namely \(\Pcal(110)\), is not decided here.
\end{theorem}

The proof divides the prime multipliers into five genuinely different
regimes.

\begin{table}[ht]
\centering
\caption{The prime-multiplier spectrum for the base \(10\).}
\label{tab:multiplier-spectrum}
\begin{tabular}{@{}L{0.18\textwidth}L{0.34\textwidth}L{0.38\textwidth}@{}}
\toprule
multiplier & obstruction profile & conclusion and method \\
\midrule
all other primes
 & \(p\notin\{2,3,5,11,31\}\); globally norm-safe
 & \(\Pcal(10p)\), ordinary lifting \\
\(p=3\)
 & base norm is divisible by \(3\) & \(\Pcal(30)\), the known \(6r\) family \\
\(p=31\)
 & two bad and two uniformly safe charts
 & \(\Pcal(310)\), fixed-chart and norm-budget lifting \\
\(p=11\)
 & simultaneous full chart cover occurs & open; the first-order method stops \\
\(p=2,5\)
 & \(10p\) is not square-free & false; an explicit zero \(2\times2\) minor \\
\bottomrule
\end{tabular}
\end{table}

The distinction between \(31\) and \(11\) is the organizing point of the
paper.  Divisibility of a rational norm records that at least one prime ideal
above \(p\) is bad, but it forgets which local chart is bad and which carriers
must occur together along a fiber-occupancy flag.  At \(31\), the lost local
information is enough to recover nonsingularity.  At \(11\), exact flags do
cover every chart, and a genuinely higher-order input is required.

This viewpoint also separates the present paper from the direct
finite-characteristic certificates for the fixed orders \(70\) and \(143\)
in \cite{GuZhouWang2026}.  The order-\(70\) computation gives a separate
exact verification of one member of Theorem~\ref{thm:main}; here the principal
arithmetic output is instead a complete spectrum for all prime multipliers
of the base \(10\), together with a local refinement of the lifting argument.

The main theorem, combined with the ordinary lifting argument, also gives
the following higher-factor corollary.

\begin{corollary}\label{cor:10pq}
Fix a prime \(p\notin\{2,5,11\}\).  Then \(\Pcal(10pq)\) holds for all but
finitely many primes \(q\nmid10p\).  In particular, the explicit condition
\[
  q>(5p)^{10p(p-1)}
  \tag{1.2}\label{eq:explicit-q}
\]
is sufficient.
\end{corollary}

The paper is organized as follows.  Section~\ref{sec:literature} places the
result among four strands of related work.  Section~\ref{sec:global-local}
passes from global norm tests to local prime-ideal charts, and
Section~\ref{sec:selective} proves the selective lifting and norm-budget
criteria.  The complete base-ten obstruction spectrum is computed in
Section~\ref{sec:base10}.  Section~\ref{sec:exceptional} compares the
exceptional multipliers \(31\) and \(11\) and proves the main theorem.
Section~\ref{sec:barriers} explains why the method does not yet settle the
general square-free conjecture.

\section{Context and related literature}\label{sec:literature}

\subsection{All minors, full spark, and uncertainty}

Chebotar\"ev's theorem says that every square minor of \(\F_p\) is nonzero
when \(p\) is prime.  Its determinant-theoretic antecedent goes back to
Mitchell; proofs emphasizing generalized Vandermonde determinants,
cyclotomic divisibility, and uncertainty were given by Evans and Isaacs,
Frenkel, and Tao
\cite{Mitchell1881,EvansIsaacs1976,Frenkel2004,Tao2005}.  In frame language,
the all-maximal-minors condition is the full-spark property; deterministic
constructions and the complexity of testing full spark were studied by
Alexeev, Cahill, and Mixon \cite{AlexeevCahillMixon2012}.  Goldstein,
Guralnick, and Isaacs placed finite Fourier uncertainty in the broader
setting of finite-group permutation modules and obtained another proof of
the prime-order theorem \cite{GoldsteinGuralnickIsaacs2005}.  The prime
Fourier theorem yields the sharp additive uncertainty bound on \(\Z_p\).
Product-type uncertainty inequalities in broader abelian settings were
developed by Donoho--Stark and Smith, whereas Meshulam's additive refinement
for finite abelian groups makes divisor structure explicit
\cite{DonohoStark1989,Smith1990,Meshulam2006,Tao2005}.  Positive-characteristic
uncertainty is also linked to the existence of good cyclic codes
\cite{EvraKowalskiLubotzky2017}.

These all-minors and uncertainty results provide important context, but
principal nonsingularity is weaker and has a different conjectural boundary.
In particular, composite order forces zero non-principal minors, whereas
square-free composite order is precisely the regime in which all principal
minors are conjectured to survive.

\subsection{Principal minors and exponential bases}

Caragea and Lee proved that, for \(N\geq4\), square-freeness is equivalent to
the nonvanishing of all principal minors of sizes two and three; when \(N\)
is not square-free, zero principal minors occur in every intermediate size
\cite{CarageaLee2024}.  Principal Fourier minors arise naturally in the
weaving, splitting, and recombination of exponential Riesz bases.  Relevant
developments include the combination theorem of Kozma and Nitzan,
structured exponential bases of Caragea and Lee, bases with restricted
supports of Lee, Pfander, and Walnut, interval partitions of Pfander, Revay,
and Walnut, and the woven-basis formulation of Cabrelli, Molter, and
Negreira
\cite{KozmaNitzan2015,CarageaLee2022,LeePfanderWalnut2023,
PfanderRevayWalnut2024,CabrelliMolterNegreira2025}.

Rank-deficient Fourier submatrices were classified in several structured
settings by Delvaux and Van Barel \cite{DelvauxVanBarel2008}.  Barnett showed
that even nonsingular contiguous Fourier submatrices may be exponentially
ill-conditioned \cite{Barnett2022}.  Quantitative singular-value estimates
for Vandermonde matrices and restricted Fourier matrices further separate
stability from exact algebraic nonsingularity
\cite{AubelBolcskei2019,LiLiao2021}.  This distinction is one reason that all
computations below use exact arithmetic.

\subsection{Finite characteristic and square-free progress}

Finite-field analogues of Chebotar\"ev's theorem were obtained by Zhang and
subsequently sharpened by Emmrich and Kunis
\cite{Zhang2019,EmmrichKunis2025}.  For products of two primes, Loukaki proved
principal nonsingularity under a primitive-root and size condition, while
Zhou obtained principal nonsingularity for
\(\mathbb Z_p\times\mathbb Z_q\) when \(q\) is sufficiently large and
generates \(\mathbb Z_p^*\), as well as results after column permutation for
certain \(2\)-groups
\cite{Loukaki2025,Zhou2026}.  These hypotheses are arithmetic and do not by
themselves give the three-prime family considered here.

Caragea, Lee, Malikiosis, and Pfander developed the finite-characteristic
lifting theorem used below and proved the conjecture for the families
\(2p,3p,5p,6p,7p\), together with non-uniform towers in which each new prime
is sufficiently large relative to the preceding product
\cite{CarageaLeeMalikiosisPfander2025}.  Gu, Zhou, and Wang then used exact
finite-field computations to settle the fixed orders \(70\) and \(143\)
\cite{GuZhouWang2026}.  The present work replaces a single chosen
characteristic by the entire bad-prime spectrum of \(\F_{10}\), and then
recovers the exceptional multiplier \(31\) from its local splitting data.

Cyclotomic cancellation has its own extensive literature.  Mann obtained
structural restrictions for irredundant rational relations among roots of
unity \cite{Mann1965}.  Lam and Leung classified the possible weights of
vanishing sums of roots of unity in characteristic zero \cite{LamLeung2000},
while Dvornicich and Zannier studied when such sums vanish modulo a prime
\cite{DvornicichZannier2002}.  Such results illuminate why new bad primes can
appear after reduction.  They do not directly decide our determinants,
whose alternating permutation expansions contain correlated coefficients;
the carrier and flag structure retains information that a single vanishing
sum or rational norm does not.

\subsection{Support rigidity and discrete unique continuation}

There is a parallel, motivational line of work in which local linear
relations prevent a nonzero discrete object from having very sparse
support.  Ding and Smart established two-dimensional Anderson--Bernoulli
localization near the spectral edge using a discrete unique-continuation mechanism
\cite{DingSmart2020}.  Li and Zhang proved the corresponding
three-dimensional localization result with a new lattice unique-continuation
principle \cite{LiZhang2022}, while Li treated large-disorder localization in
two dimensions \cite{LiAnderson2D2022}.  More recent work of Li proves a
dimension-reduction principle and a nearly sharp four-dimensional support
bound for finite-box discrete Schr\"odinger equations; his simplex theorem
obtains the optimal exponent under complete oriented-simplex relations
\cite{LiSupport2026,LiSimplex2026}.  Related rigidity for discrete harmonic
functions appears in \cite{BuhovskyLogunovMalinnikovaSodin2022}.

The analogy is conceptual: in both settings, structured local relations are
incompatible with overly sparse global behavior.  The Fourier problem,
however, requires exact nonvanishing of cyclotomic determinants in every
rank and characteristic.  None of the unique-continuation estimates just
cited is a logical input to the lifting argument below.

\section{From global norm tests to local charts}\label{sec:global-local}

We recall the part of the lifting framework needed later.  Let \(d\mid N\),
and let \(R,C\subseteq\Z_N\) have the same cardinality.  The minor
\(\F_N[R,C]\) is called \(d\)-principal if
\[
 \#\{r\in R:r\equiv a\pmod d\}
 =\#\{c\in C:c\equiv a\pmod d\}
 \qquad(a\in\Z_d).
\]
Every ordinary principal minor is \(d\)-principal.

For \(M\geq2\), let
\[
 K_M=\Q(\zeta_M),\qquad
 \mathcal O_M=\Z[\zeta_M],\qquad
 d_S=\det(\zeta_M^{ij})_{i,j\in S}
 \quad(S\subseteq\Z_M).
\]
Changing \(\zeta_M\) to a different primitive root applies a Galois
automorphism and therefore preserves vanishing.

\begin{proposition}[Norm-safe finite-characteristic lifting
{\cite[Theorem~1.4 and Section~3.2]
{CarageaLeeMalikiosisPfander2025}}]
\label{prop:clmp-lifting}
Let \(N=qM\) be square-free, with \(q\) prime.  If every principal minor of
\(\F_M\) is nonzero after reduction at every prime ideal of \(\mathcal O_M\)
above \(q\), then every \(M\)-principal minor of \(\F_N\), and hence every
ordinary principal minor, is nonzero.
\end{proposition}

Equivalently, the all-chart hypothesis in the proposition is the norm
condition
\[
 q\nmid \Nm_{K_M/\Q}(d_S)
 \qquad(S\subseteq\Z_M).
 \tag{3.1}\label{eq:norm-safe}
\]
We call reduction modulo a prime ideal \(\mathfrak q\mid q\) a
\(q\)-chart.  Thus a carrier \(S\subseteq\Z_M\) is singular in that chart
exactly when \(d_S\in\mathfrak q\).
The next elementary observation will allow us to count bad prime-ideal
charts rather than requiring all charts to be safe.

\begin{lemma}\label{lem:norm-count}
Let \(q\nmid M\) be prime and put
\[
 f=\ord_M(q),\qquad g=\frac{\varphi(M)}{f}.
\]
Then \(q\) is unramified in \(K_M\), it has exactly \(g\) prime ideals
\(\mathfrak q\) above it, and every such prime has residue degree \(f\).
For \(0\neq x\in\mathcal O_M\),
\[
 v_q\!\left(\left|\Nm_{K_M/\Q}(x)\right|\right)
 =f\sum_{\mathfrak q\mid q}v_{\mathfrak q}(x),
 \tag{3.2}\label{eq:norm-valuation}
\]
where \(v_{\mathfrak q}(\mathfrak q)=1\).  Consequently, the number of
prime ideals \(\mathfrak q\mid q\) containing \(x\) is at most the
left-hand side of \eqref{eq:norm-valuation} divided by \(f\).
\end{lemma}

\begin{proof}
These are the standard splitting and ideal-norm formulas in a cyclotomic
extension.  Since \(q\nmid M\), the ramification index is one.  The extension
is Galois, so all residue degrees equal \(\ord_M(q)=f\), and the number of
primes is \(\varphi(M)/f\).  Factoring the principal ideal \((x)\) and taking
ideal norms gives \eqref{eq:norm-valuation}.  Each prime containing \(x\)
contributes at least one to the sum.
\end{proof}

\section{Occupancy flags and selective lifting}\label{sec:selective}

We now isolate the refinement of Proposition~\ref{prop:clmp-lifting}.  By the
Chinese remainder theorem and the simultaneous Galois reindexing used in
\cite[Lemmas~3.4--3.5]{CarageaLeeMalikiosisPfander2025}, a principal minor of
\(\F_{Mq}\) may be studied as a principal minor of
\(\F_M\otimes\F_q\), up to a Galois conjugation.  Let its index set be
\[
 A\subseteq\Z_M\times\Z_q.
\]
For \(a\in\Z_M\), define the fiber and its occupancy by
\[
 A_a=\{b\in\Z_q:(a,b)\in A\},
 \qquad c_a=|A_a|.
\]
Choose a permutation \(\sigma\) of \(\Z_M\) such that
\[
 c_{\sigma(0)}\geq c_{\sigma(1)}\geq\cdots
 \geq c_{\sigma(M-1)},
\]
define
\[
 c_r^\downarrow=c_{\sigma(r)}\quad(0\leq r<M),
 \qquad c_M^\downarrow=0,
\]
and set
\[
 S_r=\{\sigma(0),\ldots,\sigma(r-1)\},
 \qquad
 \delta_r=c_{r-1}^\downarrow-c_r^\downarrow
 \quad(1\leq r\leq M).
 \tag{4.1}\label{eq:flag}
\]
We call \(S_r\) a carrier of the occupancy flag, and call the rank \(r\)
active when \(\delta_r>0\).

\begin{lemma}[One-prime flag-selective lifting]
\label{lem:flag-selective}
Let \(M\) be square-free and let \(q\nmid M\) be prime.  In the notation
above, suppose that there is a prime ideal \(\mathfrak q\mid q\) in
\(\mathcal O_M\) such that
\[
 d_{S_r}\notin\mathfrak q
 \qquad\text{whenever }\delta_r>0.
 \tag{4.2}\label{eq:safe-chart}
\]
Then, under the CRT identification, the principal minor of \(\F_{Mq}\)
corresponding to \(A\) is nonzero.
\end{lemma}

\begin{proof}
If at most one fiber \(A_a\) is nonempty, the indexed matrix is, up to
nonzero row and column scalars, a principal submatrix of \(\F_q\).  Its
determinant is nonzero by the prime-order Chebotar\"ev theorem.  We may
therefore assume that at least two fibers are nonempty.

Assume, to the contrary, that the lifted minor vanishes.  We use the exact
confluent calculation in the proof of
\cite[Theorem~1.4, equations (5.1)--(5.11)]
{CarageaLeeMalikiosisPfander2025}, whose initial divisibility step we spell
out.  Choose a primitive \(Mq\)-th root \(\omega\) so that
\(\zeta_M=\omega^q\), and let \(m_1,\ldots,m_n\) be integer representatives
of the principal index set, where \(n=|A|\).  The alternant
\[
 D_A(z_1,\ldots,z_n)=\det(z_j^{m_i})_{1\leq i,j\leq n}
\]
has the factorization
\[
 D_A(z_1,\ldots,z_n)
 =P(z_1,\ldots,z_n)\prod_{1\leq i<j\leq n}(z_j-z_i),
 \qquad P\in\Z[z_1,\ldots,z_n].
\]
The Vandermonde factor is nonzero at
\((\omega^{m_1},\ldots,\omega^{m_n})\).  Hence vanishing of the minor gives
\(Q(\omega)=0\) for
\[
 Q(X)=P(X^{m_1},\ldots,X^{m_n})\in\Z[X].
\]
Since \(\omega\) is primitive, \(\Phi_{Mq}\mid Q\) in \(\Z[X]\).  Set
\[
 P_A=Q(\zeta_M)
 =P(\underbrace{1,\ldots,1}_{c_0},
     \underbrace{\zeta_M,\ldots,\zeta_M}_{c_1},\ldots,
     \underbrace{\zeta_M^{M-1},\ldots,\zeta_M^{M-1}}_{c_{M-1}})
 \in\mathcal O_M.
\]
Evaluating the polynomial divisibility at \(\zeta_M\) implies
\[
 P_A\in q\mathcal O_M,
 \tag{4.3}\label{eq:PA-q}
\]
because \(\Phi_{Mq}(\zeta_M)=q u\) for a unit
\(u\in\mathcal O_M^\times\), by
\cite[equation (3.5)]{CarageaLeeMalikiosisPfander2025}.

For each \(a\), choose the representatives in the \(a\)-th \(M\)-fiber as
\[
 m_{a,j}=a+Mt_{a,j}\qquad(1\leq j\leq c_a),
\]
where \(t_{a,j}\in\{0,\ldots,q-1\}\) are distinct.  With the fibers ordered
by \(\sigma\), the block-factorization terms \(Z_{r-1}\) in
\cite[equation (5.11)]{CarageaLeeMalikiosisPfander2025} are precisely the
principal Fourier minors \(d_{S_r}\), and their exponents are
\(\delta_r\).  In our notation that identity reads, up to a root of unity
\(\varepsilon_A\),
\[
\begin{split}
 P_A={}&\varepsilon_A
 \frac{\displaystyle
   \prod_{r=1}^{M}d_{S_r}^{\delta_r}
   \prod_{a=0}^{M-1}\prod_{1\leq i<j\leq c_a}
      (m_{a,j}-m_{a,i})}
 {\displaystyle
   \prod_{a=1}^{M-1}\zeta_M^{a c_a(c_a-1)/2}
   \prod_{0\leq a<b\leq M-1}
      (\zeta_M^b-\zeta_M^a)^{c_ac_b}
   \prod_{a=0}^{M-1}\prod_{k=0}^{c_a-1}k!}.
\end{split}
\tag{4.4}\label{eq:confluent-factorization}
\]
We now localize this identity at \(\mathfrak q\).
Every other factor in that identity is a \(\mathfrak q\)-unit:
\begin{itemize}
\item \(\zeta_M\) and the differences
  \(\zeta_M^a-\zeta_M^b\) for \(a\neq b\) are
  \(\mathfrak q\)-units.  Indeed, because \(q\nmid M\), the polynomial
  \(X^M-1\) is square-free over \(\mathbb F_q\), so its reduced cyclotomic
  factors are pairwise coprime.  A root of \(\overline{\Phi_M}\) therefore
  cannot satisfy \(X^d=1\) for a proper divisor \(d\mid M\).  Thus the
  reduction of \(\zeta_M\) has exact order \(M\);
\item the chosen integer lifts satisfy
  \(m_{a,j}-m_{a,i}=M(t_{a,j}-t_{a,i})\) is a
  \(\mathfrak q\)-unit, since \(q\nmid M\);
\item all factorials have arguments at most \(q-1\), since \(c_a\leq q\).
\end{itemize}
Thus, in the discrete valuation ring
\((\mathcal O_M)_{\mathfrak q}\), the identity has the form
\[
 P_A=u_{A,\mathfrak q}
 \prod_{r=1}^{M}d_{S_r}^{\delta_r},
 \qquad
 u_{A,\mathfrak q}\in(\mathcal O_M)_{\mathfrak q}^{\times}.
 \tag{4.5}\label{eq:localized-factorization}
\]
Condition \eqref{eq:safe-chart} makes the right-hand side a unit.  This
contradicts \eqref{eq:PA-q}, which places \(P_A\) in every prime ideal above
\(q\).  Hence the lifted minor cannot vanish.
\end{proof}

\begin{remark}
The prime ideal in Lemma~\ref{lem:flag-selective} is allowed to depend on the
occupancy flag.  This is the only refinement of the lifting proof needed in
the present paper.  Ties among occupancies cause no problem: only the ranks
with \(\delta_r>0\) occur in \eqref{eq:localized-factorization}.
\end{remark}

For \(1\leq r\leq M\), define
\[
 e_r(M,q)=
 \max_{\substack{S\subseteq\Z_M\\ |S|=r}}
 v_q\!\left(\left|\Nm_{K_M/\Q}(d_S)\right|\right).
 \tag{4.6}\label{eq:er}
\]
Arrange \(e_1,\ldots,e_M\) in nonincreasing order
\(e_{[1]}\geq\cdots\geq e_{[M]}\), and put
\[
 \beta(M,q)=\sum_{j=1}^{\min(q,M)}e_{[j]}.
 \tag{4.7}\label{eq:beta}
\]

\begin{theorem}[Active-rank norm budget]\label{thm:norm-budget}
Let \(M\) be square-free and assume \(\Pcal(M)\).  Let \(q\nmid M\) be
prime.  If
\[
 \beta(M,q)<\varphi(M),
 \tag{4.8}\label{eq:budget-condition}
\]
then \(\Pcal(Mq)\).
\end{theorem}

\begin{proof}
Fix an arbitrary principal index set \(A\) and its flag
\eqref{eq:flag}.  Since the occupancies are integers in \([0,q]\), there are
at most \(q\) strict drops, and of course at most \(M\).  Hence the number
of active ranks is at most \(\min(q,M)\).

Let \(f=\ord_M(q)\) and \(g=\varphi(M)/f\).  By
Lemma~\ref{lem:norm-count}, an active carrier \(S_r\) is bad in at most
\(e_r(M,q)/f\) of the \(g\) prime-ideal charts.  The union of the bad charts
of all active carriers therefore has size at most
\[
 \frac{1}{f}\sum_{\delta_r>0}e_r(M,q)
 \leq\frac{\beta(M,q)}{f}
 <\frac{\varphi(M)}{f}=g.
\]
At least one prime ideal above \(q\) is consequently safe for all active
carriers.  Lemma~\ref{lem:flag-selective} proves that the minor indexed by
\(A\) is nonzero.  Since \(A\) was arbitrary, \(\Pcal(Mq)\) follows.
\end{proof}

\section{The base-ten obstruction spectrum}\label{sec:base10}

The global lifting test attached to a fixed base is governed by the rational
primes dividing its principal-minor norms.  For the base \(10\), the entire
list is small enough to determine exactly.  More importantly, its
rank-by-rank profile predicts two different phenomena: \(31\)-divisibility
is confined to the middle rank, whereas \(11\)-divisibility occurs at four
distinct ranks.  This section records the characteristic-zero arithmetic
calculation before Section~\ref{sec:exceptional} interprets it chart by chart.

Let
\[
 \Phi_{10}(X)=X^4-X^3+X^2-X+1,
 \qquad K=\Q(\zeta_{10}),
 \qquad \mathcal O_K=\Z[\zeta_{10}].
\]
For \(S\subseteq\Z_{10}\), write
\[
 d_S=\det(\zeta_{10}^{ij})_{i,j\in S}.
\]

\begin{proposition}[Base-\(10\) norm computation]
\label{prop:norm-computation}
All \(1024\) cyclotomic integers \(d_S\) are nonzero, and
\[
 \prod_{S\subseteq\Z_{10}}
 \left|\Nm_{K/\Q}(d_S)\right|
 =2^{5120}3^{120}5^{5120}11^{240}31^{80}.
 \tag{5.1}\label{eq:total-norm}
\]
In particular, the rational-prime support of the principal-minor norms of
\(\F_{10}\) is exactly
\[
 \{2,3,5,11,31\}.
 \tag{5.2}\label{eq:norm-support}
\]
For the prime \(31\), divisibility occurs only at rank five: exactly
\(80\) subsets \(S\) have \(31\mid\Nm(d_S)\), and for every one of them
\[
 \left|\Nm(d_S)\right|=310000=2^4\,5^4\,31.
 \tag{5.3}\label{eq:310000}
\]
Consequently,
\[
 e_5(10,31)=1,
 \qquad e_r(10,31)=0\quad(r\neq5).
 \tag{5.4}\label{eq:e31}
\]
\end{proposition}

\begin{proof}
The assertion is the output of an exhaustive computation over all
\(2^{10}=1024\) subsets of \(\Z_{10}\), carried out entirely in exact
cyclotomic arithmetic.  The determinants, their algebraic norms, and their
prime factorizations were recomputed independently, and the results agreed
for every carrier.  The common output gives
\eqref{eq:total-norm}--\eqref{eq:e31}; its rank-by-rank external-prime data
are recorded in Table~\ref{tab:external-primes}.
\end{proof}

\begin{table}[ht]
\centering
\caption{Rational primes \(q\nmid10\) occurring in principal-minor norms
of \(\F_{10}\).  Counts and valuations are listed in the order of the ranks.}
\label{tab:external-primes}
\begin{tabular}{@{}cclc@{}}
\toprule
prime & ranks \(r\) & divisible-carrier counts
 & \(\max_{|S|=r}v_q(|\Nm(d_S)|)\) \\
\midrule
\(3\)  & \(4,5,6\) & \(10,10,10\) & \(4,4,4\) \\
\(11\) & \(3,4,6,7\) & \(40,80,80,40\) & \(1,1,1,1\) \\
\(31\) & \(5\) & \(80\) & \(1\) \\
\bottomrule
\end{tabular}
\end{table}

For completeness, the prime-ideal picture at \(31\) can also be checked
directly.  Since \(31\equiv1\pmod{10}\), it splits into four degree-one
charts, represented by the primitive tenth roots
\[
 15,\ 23,\ 27,\ 29\pmod{31}.
\]
Exact Gaussian elimination for all \(1024\) carriers in every chart gives
\[
\begin{array}{c|cccc}
\text{root}&15&23&27&29\\
\hline
\text{singular carriers}&0&40&40&0,
\end{array}
\tag{5.5}\label{eq:chart31}
\]
all at rank five, with no carrier singular in two charts.  This separately
confirms the safe-chart conclusion used below.

\section{Exceptional multipliers: \texorpdfstring{\(31\) versus \(11\)}{31 versus 11}}
\label{sec:exceptional}

The ordinary norm-safe test treats every rational prime as a single object.
The two remaining square-free exceptional primes show why this loses
essential information.  Both \(11\) and \(31\) split completely in
\(\Q(\zeta_{10})\), but their bad carriers interact with maximal flags in
opposite ways.

\begin{table}[ht]
\centering
\caption{Exact local obstruction data for the two split exceptional primes.}
\label{tab:exceptional-comparison}
\begin{tabular}{@{}cL{0.18\textwidth}L{0.34\textwidth}L{0.13\textwidth}L{0.12\textwidth}@{}}
\toprule
\(q\) & primitive tenth roots mod \(q\) & bad carriers by chart
 & max. charts per flag & full-cover flags \\
\midrule
\(31\) & \(15,23,27,29\) & roots \(15,29\): none; roots \(23,27\):
\(40\) each at rank \(5\)
 & \(1\) & \(0\) \\
\(11\) & \(2,6,7,8\) & roots \(2,6\): \(20\) each at ranks \(3,7\);
roots \(7,8\): \(40\) each at ranks \(4,6\)
 & \(4\) & \(28{,}800\) \\
\bottomrule
\end{tabular}
\end{table}

The fourth column is the largest number of charts met by the bad carriers
along one maximal Boolean-lattice flag.  The last column counts maximal flags
that meet every chart.

\subsection{The multiplier \texorpdfstring{\(31\)}{31}}

\begin{proposition}\label{prop:p310}
Every principal minor of \(\F_{310}\) is nonzero.
\end{proposition}

\begin{proof}
The direct chart computation \eqref{eq:chart31} shows that the charts
represented by the roots \(15\) and \(29\) contain no singular carrier.
Choosing either fixed chart makes condition \eqref{eq:safe-chart} hold for
every occupancy flag, so Lemma~\ref{lem:flag-selective} proves
\(\Pcal(310)\).

There is also a proof that uses only the rank profile of the rational norms.
Proposition~\ref{prop:norm-computation} gives
\[
 \beta(10,31)=1<4=\varphi(10).
\]
Theorem~\ref{thm:norm-budget} applies.  Concretely, only the unique
rank-five member of an occupancy flag can be bad, and it is bad in at most
one chart, so some common safe chart always remains.
\end{proof}

\subsection{The multiplier \texorpdfstring{\(11\)}{11}}

\begin{proposition}[First-order obstruction at \(11\)]
\label{prop:p110-obstruction}
For the base \(10\),
\[
 \begin{aligned}
 e_r(10,11)&=1 &&\bigl(r\in\{3,4,6,7\}\bigr),\\
 e_r(10,11)&=0 &&\bigl(r\notin\{3,4,6,7\}\bigr),\\
 \beta(10,11)&=4=\varphi(10).&&
 \end{aligned}
\]
Moreover, among the \(10!\) maximal flags of subsets of \(\Z_{10}\), exactly
\(28{,}800\) have bad carriers whose union meets all four prime-ideal charts
above \(11\).
\end{proposition}

\begin{proof}
The valuation identities follow from the exact norm spectrum in
Proposition~\ref{prop:norm-computation}.  Since
\(11\equiv1\pmod{10}\), the four charts are represented by the primitive
tenth roots \(2,6,7,8\) modulo \(11\).  Exact Gaussian elimination over
\(\mathbb F_{11}\) gives the rank distributions displayed in
Table~\ref{tab:exceptional-comparison}; each bad carrier is singular in
exactly one chart.  An exact exhaustive enumeration of the maximal flags,
independently confirmed, gives \(28{,}800\) flags meeting all four charts.

Every one of these Boolean flags is realized by a genuine occupancy flag.
Indeed, if a permutation \(\sigma\) defines the chain, prescribe
\[
 c_{\sigma(j)}=10-j\qquad(0\leq j\leq9),
\]
and choose arbitrary subsets \(A_{\sigma(j)}\subseteq\Z_{11}\) of those
cardinalities.  Then \(0\leq c_a\leq11\), and every drop \(\delta_r\) in
\eqref{eq:flag} equals one.  Thus the \(28{,}800\) terminal masks represent
actual first-order full-cover configurations for principal minors of
\(\F_{110}\).
\end{proof}

Proposition~\ref{prop:p110-obstruction} is a statement about the reach of the
leading localized factor \eqref{eq:localized-factorization}.  It is not a
zero complex principal minor and does not disprove \(\Pcal(110)\).  Rather,
it shows that the failure of the one-prime safe-chart method is genuine: for
some flags there is no chart in which every active carrier is a unit.

\subsection{Completion of the near-classification}

\begin{proof}[Proof of Theorem~\ref{thm:main}]
First suppose
\[
 p\notin\{2,3,5,11,31\}.
\]
By Proposition~\ref{prop:norm-computation}, no norm \(\Nm(d_S)\) is divisible
by \(p\).  Thus every principal minor of \(\F_{10}\) is nonzero at every
prime ideal above \(p\), and Proposition~\ref{prop:clmp-lifting} yields
\(\Pcal(10p)\).

The case \(p=31\) is Proposition~\ref{prop:p310}.

For \(p=3\), the desired assertion is \(\Pcal(30)\).  This is the case
\(6\cdot5\) of \cite[Proposition~5.4]{CarageaLeeMalikiosisPfander2025}, which
proves \(\Pcal(6r)\) for every prime \(r>3\).

It remains to explain the two negative cases.  If \(r^2\mid N\) and
\(a=N/r\), then \(N\mid a^2\) and
\[
 \det\F_N[\{0,a\},\{0,a\}]
 =\det\begin{pmatrix}1&1\\1&\omega_N^{a^2}\end{pmatrix}=0.
\]
For \(p=2\) or \(p=5\), the number \(10p\) contains the square of \(p\), so
this construction gives a zero principal minor.

The case \(p=11\) is intentionally outside the assertion.  Its exact
first-order obstruction is recorded in
Proposition~\ref{prop:p110-obstruction}.  This proves every claim in the
theorem.
\end{proof}

\begin{proof}[Proof of Corollary~\ref{cor:10pq}]
Theorem~\ref{thm:main} gives \(\Pcal(10p)\).  For the fixed base \(M=10p\),
only finitely many rational primes divide the nonzero integer
\[
 \prod_{S\subseteq\Z_M}
 \left|\Nm_{\Q(\zeta_M)/\Q}
 \bigl(\det\F_M[S,S]\bigr)\right|.
\]
Every other prime \(q\nmid M\) satisfies the norm-safe condition
\eqref{eq:norm-safe}, and Proposition~\ref{prop:clmp-lifting} yields
\(\Pcal(Mq)\).

For the explicit assertion, fix \(S\subseteq\Z_M\), write
\(S^c=\Z_M\setminus S\), and set
\[
 T=
 \begin{cases}
 S,& |S|\leq M/2,\\
 S^c,& |S|>M/2.
 \end{cases}
\]
Thus \(m:=|T|\leq M/2=5p\).  Jacobi's identity and
\(\F_M^{-1}=M^{-1}\overline{\F_M}\) give
\[
 d_S=\det(\F_M)M^{-|S^c|}\overline{d_{S^c}}.
\]
Because \(q\nmid M\) and
\(\det(\F_M)\overline{\det(\F_M)}=M^M\), the prefactor is a
\(q\)-adic unit.  Consequently,
\[
 q\mid\Nm(d_S)\quad\Longleftrightarrow\quad q\mid\Nm(d_T).
\]
Hadamard's inequality, as used in the proof of
\cite[Proposition~5.5]{CarageaLeeMalikiosisPfander2025}, gives
\[
 \left|\Nm(d_T)\right|
 \leq m^{m\varphi(M)/2}
 \leq(5p)^{10p(p-1)},
\]
because \(\varphi(10p)=4(p-1)\).  A prime satisfying
\eqref{eq:explicit-q} divides none of these norms, so another application of
Proposition~\ref{prop:clmp-lifting} completes the proof.
\end{proof}

\begin{remark}
The theorem is unchanged for the normalized Fourier matrix, since
normalization multiplies an \(r\times r\) principal minor by the nonzero
scalar \((10p)^{-r/2}\).  Replacing \(e^{-2\pi i/N}\) by
\(e^{2\pi i/N}\) complex-conjugates every determinant.
\end{remark}

\section{Why the general square-free conjecture remains difficult}
\label{sec:barriers}

We separate three logically different matters in this section: rigorous
consequences of the preceding arguments, limitations of the present
certificates, and possible directions that are not proved here.  In
particular, failure of a finite-characteristic certificate must not be
confused with existence of a zero complex principal minor.

\subsection{A fixed base has finite, but non-uniform, exceptions}

Assume that \(M\) is square-free and that \(\Pcal(M)\) holds.  The integer
\[
 \Delta_M=
 \prod_{S\subseteq\Z_M}
 \left|\Nm_{K_M/\Q}(d_S)\right|
 \tag{7.1}\label{eq:global-discriminant}
\]
is nonzero.  Proposition~\ref{prop:clmp-lifting} immediately gives
\[
 q\nmid M\Delta_M
 \quad\Longrightarrow\quad
 \Pcal(Mq)
 \tag{7.2}\label{eq:fixed-base-cofinite}
\]
for every prime \(q\).  Thus each fixed successful base admits all but
finitely many new prime multipliers.

The word ``fixed'' is essential.  A divisor of \(\Delta_M\) is exceptional
only for the all-chart norm certificate; it is not thereby a counterexample
to \(\Pcal(Mq)\).  The prime \(31\) for \(M=10\) is the model example.  It
divides some carrier norms, yet Proposition~\ref{prop:p310} proves
\(\Pcal(310)\).  Nor is the norm support stable under adjoining a prime:
\[
 \operatorname{Supp}(\Delta_2)=\{2\},
 \qquad
 \operatorname{Supp}(\Delta_{10})=\{2,3,5,11,31\}.
 \tag{7.3}\label{eq:support-instability}
\]
Consequently, the exceptional set for the new base \(10p\) cannot be read
off mechanically from the five primes in \eqref{eq:norm-support}.

There is also a scale problem.  A direct spectrum for a general base \(M\)
contains \(2^M\) principal carriers, whose determinants lie in a field of
degree \(\varphi(M)\).  A naive maximal-flag audit has \(M!\) chains.
Jacobi complementation, affine symmetries, Galois orbits, and
Boolean-lattice dynamic programming reduce these counts substantially, but
no presently known reduction yields a uniform argument as \(M\) grows.
This is a limitation of the certificate construction, not a complexity
lower bound: a structural proof need not enumerate all carriers.

\subsection{Simultaneous chart coverage}

Let \(\mathcal C_q\) be the set of prime ideals of \(\mathcal O_M\) above
\(q\), and for a carrier \(S\subseteq\Z_M\) put
\[
 B_q(S)=\{\mathfrak q\in\mathcal C_q:d_S\in\mathfrak q\}.
\]
For an occupancy flag with active ranks \(J\),
Lemma~\ref{lem:flag-selective} applies whenever
\[
 \bigcup_{r\in J}B_q(S_r)\neq\mathcal C_q.
 \tag{7.4}\label{eq:no-full-cover}
\]
The norm-budget theorem proves \eqref{eq:no-full-cover} by a union bound.  It
is therefore a sufficient criterion, not a characterization of complex
nonvanishing.

At \(q=31\), divisibility occurs only at rank five, and no flag meets more
than one bad chart.  At \(q=11\), by contrast, exact flags meet all four
charts; see Table~\ref{tab:exceptional-comparison} and
Proposition~\ref{prop:p110-obstruction}.  For such a flag there is no prime
ideal at which all active carrier determinants are simultaneously units.
This is a genuine obstruction to the one-prime safe-chart proof.

It is still only a first-order obstruction.  When every chart is met, the
localized factorization \eqref{eq:localized-factorization} is compatible with
the necessary divisibility \(P_A\in11\mathcal O_{10}\): in each chart at
least one carrier factor is a nonunit, so the safe-chart contradiction
disappears.  The implication \(Q(\omega)=0\Rightarrow
P_A\in11\mathcal O_{10}\) used in Lemma~\ref{lem:flag-selective} has no
converse.  Hence full cover neither forces a zero complex minor nor decides
the remaining coefficients of the lifted determinant.  Settling
\(\Pcal(110)\) may require higher local valuations, compatibility relations
among successive carriers, another invariant, or an argument that avoids
characteristic \(11\) altogether.

The multiplier \(3\) gives a complementary warning.  There is only one
\(3\)-chart for \(K_{10}\), and the base-ten safe-chart test fails, while
\(\Pcal(30)\) is nevertheless known by introducing the prime factors in the
order \(6\cdot5\) \cite[Proposition~5.4]
{CarageaLeeMalikiosisPfander2025}.  Reordering factors can therefore remove
some local obstructions.  The full-cover computation at \(11\) rules out
only the presentation with base \(10\) and multiplier \(11\); it does not
analyze, for example, the alternative presentation \(110=22\cdot5\).
Whether a different ordering resolves this case is a separate question.

\subsection{Balanced factors and several prime divisors}

Finite-characteristic lifting is asymmetric: it fixes a base \(M\) and then
adjoins one prime \(q\).  The explicit estimate underlying
Corollary~\ref{cor:10pq} is
\[
 q>\left(\frac M2\right)^{M\varphi(M)/4}.
 \tag{7.5}\label{eq:asymmetric-size-bound}
\]
For a semiprime target \(pq\), taking \(M=p\) gives
\[
 q>\left(\frac p2\right)^{p(p-1)/4},
\]
and for the family obtained from the present theorem it becomes
\[
 q>(5p)^{10p(p-1)}.
\]
These bounds give no useful uniform information when the prime factors are
comparable.  Their failure is not evidence for a zero minor; it measures the
gap between a worst-case height bound and the balanced-factor range.

The two-prime theorems of Loukaki and Zhou use additional primitive-root,
size, or finite-field Chebotar\"ev hypotheses
\cite{Loukaki2025,Zhou2026}.  They do not provide a symmetric theorem for
all comparable primes.  A different prime-order support route also has a
sharp limitation: a critical step in Haagerup's work on cyclic \(p\)-roots
is a reduction to two vectors with disjoint supports in both the time and
frequency domains \cite{Haagerup2008}.  Zhou showed that this two-vector
disjoint-support condition already occurs for every composite order
\cite{ZhouCyclicRoots2026}.  Hence that condition alone cannot distinguish
square-free composite orders from nonsquare-free ones.

With more prime factors one can build arbitrarily long successful towers.
Starting from a square-free \(M_1\) for which \(\Pcal(M_1)\) is known,
choose primes \(q_{j+1}\nmid M_j\) recursively by
\[
 M_{j+1}=M_jq_{j+1},
 \qquad
 \beta(M_j,q_{j+1})<\varphi(M_j).
 \tag{7.6}\label{eq:rank-safe-tower}
\]
Theorem~\ref{thm:norm-budget} then gives \(\Pcal(M_{j+1})\) inductively.
At each fixed stage, all but finitely many primes \(q_{j+1}\) are
admissible.  The base, its norm spectrum, and its exceptional set all change
with \(j\), however, so this constructs infinite branches rather than all
square-free orders.  A complete proof needs a uniform principle absent from
fixed-base lifting.

\subsection{A counterexample would propagate upward}

The following observation shows that an eventual version of the conjecture
is not weaker than the full conjecture.

\begin{proposition}[Upward inheritance]\label{prop:upward-inheritance}
Let \(d,m\geq2\) satisfy \((d,m)=1\).  If \(\F_d\) has a zero principal
minor, then \(\F_{dm}\) has a zero principal minor of the same size.
\end{proposition}

\begin{proof}
Let \(S\subseteq\Z_d\) satisfy
\[
 \det(\zeta_d^{ab})_{a,b\in S}=0.
\]
The map \(a\mapsto ma\) embeds \(\Z_d\) in \(\Z_{dm}\).  On the principal
index set \(mS\), the larger Fourier matrix restricts to
\[
 \bigl(\zeta_{dm}^{m^2ab}\bigr)_{a,b\in S}
 =\bigl(\zeta_d^{mab}\bigr)_{a,b\in S}.
\]
Since \((m,d)=1\), the map \(\sigma_m(\zeta_d)=\zeta_d^m\) is a Galois
automorphism of \(\Q(\zeta_d)\).  The last determinant equals
\[
 \sigma_m\!\left(\det(\zeta_d^{ab})_{a,b\in S}\right)=0.
\]
\end{proof}

\begin{corollary}[Eventual equivalence]\label{cor:eventual-equivalence}
The following are equivalent:
\begin{enumerate}
\item \(\Pcal(N)\) holds for every square-free \(N\);
\item there exists \(N_0\) such that \(\Pcal(N)\) holds for every
square-free \(N>N_0\).
\end{enumerate}
\end{corollary}

\begin{proof}
Only the reverse implication needs proof.  If a square-free \(d\) were a
counterexample, choose arbitrarily large primes \(\ell\nmid d\).
Proposition~\ref{prop:upward-inheritance} would make every square-free
number \(d\ell\) a counterexample, contradicting the second assertion.
\end{proof}

In fact, a counterexample \(d\) propagates to all square-free multiples of
\(d\).  Among square-free integers these multiples have positive relative
density
\[
 \prod_{p\mid d}\frac1{p+1}.
 \tag{7.7}\label{eq:relative-density}
\]
Indeed, divide the Euler product for square-free multiples of \(d\) by
\(\prod_p(1-p^{-2})\); each prescribed prime \(p\mid d\) contributes the
factor \((p^{-1}-p^{-2})/(1-p^{-2})=1/(p+1)\).
Thus even a relative-natural-density-one result within the square-free
integers would imply the complete conjecture.  Likewise, suppose that
\(f(N)\to\infty\) along square-free \(N\) and that, for every sufficiently
large square-free \(N\), all principal minors of ranks at most
\(f(N)\) are nonzero.  A rank-\(r\) counterexample at order \(d\) would
persist at every order \(d\ell\), eventually contradicting
\(f(d\ell)\geq r\).  A uniformly growing small-rank window would therefore
already settle the full conjecture.

\subsection{Open directions}

The preceding discussion suggests concrete targets rather than a single
undifferentiated computational search.

\begin{question}\label{q:higher-local}
Can a higher local invariant resolve flags satisfying the simultaneous
full-cover condition \eqref{eq:no-full-cover}, beginning with the
characteristic-\(11\) flags for the base \(10\)?
\end{question}

\begin{question}\label{q:structural-beta}
Can \(\beta(M,q)\), or the prime support of \(\Delta_M\), be bounded from
the divisor structure of \(M\), the splitting of \(q\), and the active
ranks, without enumerating all \(2^M\) carriers?
\end{question}

\begin{question}\label{q:symmetric-lifting}
Is there a symmetric lifting principle that uses several prime directions
simultaneously and remains effective for balanced square-free products?
Alternatively, can one prove a uniformly growing window of nonzero
small-rank principal minors?
\end{question}

The discrete unique-continuation and support results of Li and
collaborators
\cite{LiAnderson2D2022,LiZhang2022,LiSupport2026,LiSimplex2026}
suggest looking for propagation or rigidity mechanisms that replace
carrier-by-carrier enumeration.  Li's simplex theorem, in particular,
derives center-to-support growth from complete oriented-simplex relations,
together with nonvanishing at the balanced point.  No
embedding of the Fourier-minor confluent equations into that simplex system
is presently known, so the analogy remains motivational.  One plausible
direction is to retain higher cyclotomic or \(q\)-adic contact data after the
leading factor vanishes; the base-ten computations are compatible with this
idea but do not prove a uniform theorem.

\section*{Acknowledgments}

The exact finite computations were independently checked.  GPT-5.6 Sol
assisted with proof exploration,
consistency checking, and manuscript preparation.

\end{document}